\documentclass[12pt]{article}
\usepackage{array,amsmath,amssymb,amsthm}  				    
\usepackage{fullpage,lineno}

\usepackage{color}

\begin{document}

\newtheorem{theorem}{Theorem}
\newtheorem{lemma}[theorem]{Lemma}
\newtheorem{observation}[theorem]{Observation}
\newtheorem{corollary}[theorem]{Corollary}
\newtheorem{prop}[theorem]{Proposition}
\newtheorem{conjecture}[theorem]{Conjecture}
\newtheorem{claim}[theorem]{Claim}
\theoremstyle{definition}
\newtheorem{defn}[theorem]{Definition}
\newtheorem{remark}[theorem]{Remark}
\newtheorem{alg}[theorem]{Algorithm}
\def\qed{\ifhmode\unskip\nobreak\hfill$\Box$\bigskip\fi \ifmmode\eqno{Box}\fi}

\def\nul{\varnothing} 
\def\st{\colon\,}   
\def\VEC#1#2#3{#1_{#2},\ldots,#1_{#3}}
\def\VECOP#1#2#3#4{#1_{#2}#4\cdots #4 #1_{#3}}
\def\SE#1#2#3{\sum_{#1=#2}^{#3}} 
\def\PE#1#2#3{\prod_{#1=#2}^{#3}}
\def\UE#1#2#3{\bigcup_{#1=#2}^{#3}}
\def\CH#1#2{\binom{#1}{#2}} 
\def\FR#1#2{\frac{#1}{#2}}
\def\FL#1{\left\lfloor{#1}\right\rfloor} \def\FFR#1#2{\FL{\frac{#1}{#2}}}
\def\CL#1{\left\lceil{#1}\right\rceil}   \def\CFR#1#2{\CL{\frac{#1}{#2}}}
\def\Gb{\overline{G}}
\def\NN{{\mathbb N}} \def\ZZ{{\mathbb Z}} \def\QQ{{\mathbb Q}}
\def\RR{{\mathbb R}} \def\GG{{\mathbb G}} \def\FF{{\mathbb F}}

\def\B#1{{\bf #1}}      \def\R#1{{\rm #1}}
\def\I#1{{\it #1}}      \def\c#1{{\cal #1}}
\def\C#1{\left | #1 \right |}    
\def\P#1{\left ( #1 \right )}    
\def\ov#1{\overline{#1}}        \def\un#1{\underline{#1}}

\expandafter\ifx\csname dplus\endcsname\relax \csname newbox\endcsname\dplus\fi
\expandafter\ifx\csname dplustemp\endcsname\relax
\csname newdimen\endcsname\dplustemp\fi
\setbox\dplus=\vtop{\vskip -8pt\hbox{%
    \kern .2em
    \special{pn 6}%
    \special{pa 0 50}%
    \special{pa 50 100}%
    \special{fp}%
    \special{pa 50 100}%
    \special{pa 100 50}%
    \special{fp}%
    \special{pa 100 50}%
    \special{pa 50 0}%
    \special{fp}%
    \special{pa 50 100}%
    \special{pa 50 0}%
    \special{fp}%
    \special{pa 50 0}%
    \special{pa 0 50}%
    \special{fp}%
    \special{pa 0 50}%
    \special{pa 100 50}%
    \special{fp}%
    \hbox{\vrule depth0.080in width0pt height 0pt}%
    \kern .7em
  }%
}%
\def\gjoin{\copy\dplus}

\def\cD{{\mathcal D}}
\def\e{{\rm e}}
\def\la{\langle}
\def\ra{\rangle}
\def\symd{\kern-.1ex{\triangle}\kern.3ex}
\long\def\skipit#1{}


\title{$2$-Reconstructibility of Weakly Distance-Regular Graphs}

\author{
Douglas B. West\thanks{Zhejiang Normal Univ., Jinhua, China
and Univ.\ of Illinois at Urbana--Champaign, Urbana IL:
\texttt{dwest@illinois.edu}.  Supported by National Natural Science Foundation
of China grant NSFC 11871439, 11971439, and U20A2068.}\,,
Xuding Zhu\thanks{Zhejiang Normal Univ., Jinhua, China:
\texttt{xdzhu@zjnu.edu.cn}.  Supported by National Natural Science Foundation
of China grant NSFC 11971438 and U20A2068 and by Zhejiang Natural Science
Foundation grant ZJNSF LD19A010001.}
}

\maketitle

\baselineskip 16pt
\vspace{-2pc}

\begin{abstract}
A graph is $\ell$-reconstructible if it is determined by its multiset
of induced subgraphs obtained by deleting $\ell$ vertices.  We prove that
strongly regular graphs with at least six vertices are $2$-reconstructible.
\end{abstract}

The {\it $k$-deck} of an $n$-vertex graph is the multiset of its $\CH nk$
induced subgraphs with $k$ vertices.  The famous Reconstruction Conjecture of
Ulam~\cite{Kel1,Ulam} asserts that when $n\ge3$, every $n$-vertex graph is
determined by its $(n-1)$-deck.  One can consider more generally whether an
$n$-vertex graph is determined by its $(n-\ell)$-deck.  A graph or graph
property is {\it $\ell$-reconstructible} if it is determined by the deck
obtained by deleting $\ell$ vertices.  In light of the following observation,
we seek the maximum $\ell$ such that a graph is $\ell$-reconstructible.
The observation holds because each card in the $k'$-deck appears as an induced
subgraph in the same number of cards in the $k$-deck.

\begin{observation}\label{decks}
For $k'<k$, the $k$-deck of a graph determines the $k'$-deck.
\end{observation}

In light of this observation, Manvel~\cite{M1,M2} posed a more general version
of the Reconstruction Conjecture, and he called this more general version
``Kelly's Conjecture''.

\begin{conjecture}[{\rm\cite{M1,M2}}]\label{manconj}
For each natural number $\ell$, there is a threshold $M_\ell$ such that every
graph with at least $M_\ell$ vertices is $\ell$-reconstructible.
\end{conjecture}

The original Reconstruction Construction is $M_1=3$.  Since the graph
$C_4+K_1$ and the tree $K_{1,3}'$ obtained by subdividing one edge of $K_{1,3}$
have the same $3$-deck, $M_2\ge6$.  Since $P_{2\ell}$ and
$C_{\ell+1}+P_{\ell-1}$ have the same $\ell$-deck (\cite{SW}), in general
$M_\ell\ge2\ell+1$, and a difficult result of N\'ydl~\cite{N1} implies that
$M_\ell$, if it exists, must grow superlinearly.  (We use $C_n,P_n,K_n$ for the
cycle, path, and complete graph with $n$ vertices, $K_{r,s}$ for the complete
bipartite graph with parts of sizes $r$ and $s$, and $G+H$ for the disjoint
union of graphs $G$ and $H$.  Our graphs have no loops or multi-edges.)

Kostochka and West~\cite{KW} surveyed results on $\ell$-reconstructibility of
graphs, so we do not provide an exhaustive review here.  One theme is to prove
that graphs in a particular family are $\ell$-reconstructible.  We consider
$2$-reconstructibility of a special family of regular graphs, where a graph is
{\it regular} if all vertices have the same degree.  One of the first results
about reconstruction is that regular graphs having at least three vertices are
$1$-reconstructible (Kelly~\cite{Kel2}).  By Observation~\ref{decks}, the
$(n-1)$-deck of a graph $G$ determines the $2$-deck and hence the number of
edges, so in each card of the $(n-1)$-deck we know the degree of the missing
vertex.  We then know we have a $k$-regular graph, and in any card the
neighbors of the missing vertex are those with degree $k-1$ in the card.

Bojan Mohar asked whether regular graphs are $2$-reconstructible.  
Although Chernyak~\cite{Che} proved that the degree list is $2$-reconstructible
for graphs with at least six vertices ($C_4+K_1$ and $K_{1,3}'$ show that
this is sharp), knowing that the graph is $k$-regular does not generally
provide enough information to decide which of the deficient vertices in a card
adjacent to which of the two vertices missing from the card (those of degree
$k-2$ in the card are adjacent to both missing vertices).  Nevertheless,
Kostochka, Nahvi, West, and Zirlin~\cite{KNWZ1} proved that $3$-regular graphs
are $2$-reconstructible.

In this note, we consider regular graphs of higher degree but restrict the
structure of common neighbors.  A graph is {\it strongly regular} with
parameters $(k,\lambda,\mu)$ if it is $k$-regular, every two adjacent vertices
have exactly $\lambda$ common neighbors, and every two nonadjacent vertices
have exactly $\mu$ common neighbors.  Discussion of strongly regular graphs and
their properties can be found for example in the book by van Lint and
Wilson~\cite{vLW}.

Most of our argument applies to graphs in a more general family.
A graph is {\it distance-regular} if for any two vertices $u$ and
$v$, the number of vertices at distance $i$ from $u$ and distance $j$ from $v$
depends only on $i$, $j$, and the distance between $u$ and $v$.  For graphs
with diameter $d$, an equivalent condition is the existence of parameters
$(\VEC b0d;\VEC c0d)$ (called the {\it intersection array} of $G$) such that
for all $u,v\in V(G)$ separated by distance $j$, the numbers of neighbors of
$u$ having distance $j+1$ or $j-1$ from $v$ are $b_j$ and $c_j$, respectively
(Brouwer et al.\cite{BCN}).  A strongly regular graph with parameters
$(k,\lambda,\mu)$ is distance-regular with intersection array
$(k,k-\lambda-1,0;0,1,\mu)$.  In fact, a non-complete distance-regular graph is
strongly regular if and only if it has diameter $2$ (Biggs~\cite{Big}).

A disjoint union of complete graphs with at least six vertices is 
$2$-reconstructible, because we know the degree list and we know that
no three vertices induce $P_3$.  Also, connectedness of an $n$-vertex
graph is determined by the $(n-2)$-deck when $n\ge6$ (Manvel~\cite{M2}).
Hence in our discussion we may assume that we are given the deck of an
$n$-vertex connected graph.  We will prove $2$-reconstructibility for all
strongly regular graphs and all graphs in a family that includes all
distance-regular graphs where vertices at distance $2$ have at least two common
neighbors.  The latter family includes all distance-regular graphs.

We define a regular graph to be {\it weakly distance-regular} if 
any two adjacent vertices have $\lambda$ common neighbors and any
two vertices separated by distance $2$ have $\mu'$ common neighbors.
In particular, we are requiring the existence of only one of the parameters
of distance-regular graphs for nonadjacent vertices.
We prove that strongly regular graphs (even with $\mu=1$) and
weakly distance-regular graphs with $\mu'\ge2$ are $2$-reconstructible.

For strongly regular graphs,
the method is analogous to the proof of $1$-reconstructibility of regular
graphs.  We use all the cards in the $(n-2)$-deck to recognize that any
graph having this deck is strongly regular and to determine the parameters
$(k,\lambda,\mu)$.  We then use a single card to reconstruct the graph.

\begin{theorem}\label{srg}
Strongly regular graphs with at least six vertices are $2$-reconstructible.
\end{theorem}
\begin{proof}
Let $G$ be an $n$-vertex graph, where $n\ge6$, and let $\cD$ be the
$(n-2)$-deck of $G$.  By the result of Chernyak~\cite{Che}, $\cD$ determines
the degree list of $G$ and hence whether $G$ is $k$-regular.  If so, then any
card $C$ in $\cD$ is missing $2k-1$ or $2k$ of the $kn/2$ edges in $G$,
depending on whether the two omitted vertices are adjacent or not.  Hence we
also see whether the vertices omitted by $C$ are adjacent.  Their number of
common neighbors is the number of vertices with degree $k-2$ in $C$.  The
graph $G$ is strongly regular with parameters $(k,\lambda,\mu)$ if and only if
that number is $\lambda$ in each card missing $2k-1$ edges and $\mu$ in each
card missing $2k$ edges.

Having recognized that $G$ is strongly regular with parameters
$(k,\lambda,\mu)$, consider one card $C$, and let $u$ and $v$ be the two
omitted vertices.  We know whether $u$ and $v$ are adjacent.  If $G$ is not
$K_n$, which we can determine, then we may choose $C$ so that $u$ and $v$
are not adjacent.  We know the $\mu$ common neighbors of $u$ and $v$, and
we know the set $S$ of $2k-2\mu$ vertices that are adjacent to exactly one of
$\{u,v\}$.

For $x,y\in S$, each of $x$ and $y$ has one neighbor in $\{u,v\}$; the
neighbors may be the same or distinct.  The vertices $x$ and $y$ have $\lambda$
or $\mu$ common neighbors in $G$, depending on whether they are adjacent.  We
see in $C$ whether they are adjacent, so we know their number of common
neighbors in $G$; call it $\rho$.  If $x$ and $y$ have $\rho$ common neighbors
in $C$, then they have different neighbors in $\{u,v\}$; if they have $\rho-1$
common neighbors in $C$, then they have the same neighbor in $\{u,v\}$.

This labels each pair of vertices in $S$ as ``same'' or ``different''.
Also, the relation defined by ``same'' is an equivalence relation.  Hence
it partitions $S$ into two sets.  We assign one of those sets to the
neighborhood of $u$ and the other to the neighborhood of $v$.  It does not
matter which set we assign to which neighborhood, because in both cases we
obtain the same graph, and it is $G$.
\end{proof}

The proof of Theorem~\ref{srg} applies to all connected strongly regular
graphs.  In particular, we allow the possibility $\mu=1$.  For the more general
class of weakly distance-regular graphs, we need to work harder, and the
proof does not apply to the case $\mu'=1$.

\begin{theorem}\label{wdr}
Weakly distance-regular graphs with at least six vertices and parameters
$(k,\lambda,\mu')$ with $\mu'\ge2$ are $2$-reconstructible.
\end{theorem}
\begin{proof}
As in the proof of Theorem~\ref{srg}, we know the degree list and thus can
recognize both that $G$ is $k$-regular and whether the missing vertices in
any card are adjacent in $G$.  The number of common neighbors of the two
missing vertices in a card is the number of vertices having degree $k-2$ in
the card.  To recognize that $G$ is in the specified class, we check that 
these numbers all equal $\lambda$ when the missing vertices are adjacent
and equal $\mu'$ when the missing vertices are nonadjacent and the number is
positive.  The number is $0$ when the distance between the missing vertices
in $G$ exceeds $2$.  Hence we can recognize that $G$ is weakly distance-regular
with parameters $(k,\lambda,\mu')$ (including when $\mu'=1$).

Given a card $C$, again let $S$ be the set of vertices adjacent to exactly one
of the two vertices $u$ and $v$ missing from $C$; these are the vertices
having degree $k-1$ in $C$.  Let $x$ and $y$ be two vertices in $S$.
We see in $C$ whether $x$ and $y$ are adjacent.  If so, then they have
$\lambda$ common neighbors in $G$.  Their number of common neighbors in $C$
is then $\lambda$ or $\lambda-1$, which tells us whether they have the same
neighbor in $\{u,v\}$.

Since $\mu'\ge2$, when $x$ and $y$ are nonadjacent in $G$ we see a common
neighbor of $x$ and $y$ in $C$ if and only if the distance between $x$ and $y$
in $G$ is $2$.  Hence for the pairs of vertices in $S$ separated by distance
$2$, we can again tell whether their neighbors in $\{u,v\}$ are the same or
different.  The pairs of vertices in $S$ that are separated by distance more
than $2$ in $G$ are those having no common neighbor in $C$.  They must have
distinct neighbors in $\{u,v\}$, and the distance between them is $3$.

With these arguments, we know for all pairs of vertices in $S$ whether their
neighbors in $\{u,v\}$ are the same or different.  Hence again we have two
equivalence classes and assign one class to the neighborhood of each of these
vertices to complete the reconstruction of $G$.
\end{proof}

\begin{center}
{\Large\bf Acknowledgment}
\end{center}

We thank Alexandr V. Kostochka for helpful discussions.

\end{document}